\documentclass[reqno]{amsart}
\usepackage{amssymb}
\usepackage{hyperref}
\usepackage[sort,nocompress]{cite} 
\newtheorem{theorem}{Theorem}[section]
\newtheorem{lemma}[theorem]{Lemma}
\usepackage{amsaddr}
\usepackage{graphicx}
\theoremstyle{definition}

\theoremstyle{remark}
\newtheorem{remark}[theorem]{Remark}

\numberwithin{equation}{section}
\makeatletter
\renewcommand\subsection{\@startsection{subsection}{2}%
	\z@{-.5\linespacing\@plus-.7\linespacing}{.5\linespacing}%
	{\normalfont\scshape}}
\renewcommand\subsubsection{\@startsection{subsubsection}{3}%
	\z@{.5\linespacing\@plus.7\linespacing}{-.5em}%
	{\normalfont\scshape}}
\makeatother

\begin{document}

\title[High-order Compact FDM for Acoustic Wave Equations]{An Efficient and  high accuracy  finite-difference scheme for the acoustic wave equation in 3D heterogeneous media}

\author{Keran Li, Wenyuan Liao}
\address{Department of Mathematics and Statistics, University of Calgary, AB, T2N 1N4, Canada}
\email{keran.li1@ucalgary.ca}

\email{wliao@ucalgary.ca}





\begin{abstract}
Efficient and accurate numerical simulation of  3D acoustic wave propagation in heterogeneous media plays an important role in 
the success of seismic full waveform  inversion (FWI) problem. 
In this work we employed the combined scheme  and developed a new explicit compact high-order finite difference scheme to solve the 3D  acoustic wave equation with spatially variable acoustic velocity. The boundary conditions  for the second derivatives of spatial variables have been derived by using
the equation itself and the boundary  condition for $u$. Theoretical analysis shows that the new scheme has an accuracy order of $O(\tau^2) + O(h^4)$, where $\tau$ is the time step and $h$ is the grid size. Combined with Richardson extrapolation or Runge-Kutta method, the new method can be improved to 4th-order in time. Three numerical experiments are  conducted to validate the efficiency and accuracy of the new scheme. The stability of the new scheme has been proved by an energy method, which shows that the new scheme is conditionally stable with a Courant - Friedrichs - Lewy (CFL) number which is slightly lower than that of the Pad\'{e} approximation based method.

\smallskip
\noindent \textbf{Keywords } Acoustic Wave Equation, Compact Finite Difference, High-order Scheme, Heterogenous Media.
\end{abstract}

\maketitle

\section{Introduction}
Finite Difference Methods (FDM) have been extensively applied in both science and engineering models, especially in solving partial differential equations numerically, when the analytical solution is not available. For example,  the acoustic wave equation with a non-zero point source function 
has been widely used to model the wave propagation in Geophysics. These acoustic wave equations widely arise in various applications including underground imaging, seabed exploration, etc.  The finite difference method is one of the popular choices for solving  these models, and the efficiency and accuracy of the finite difference method are critical, especially when the problem is in large size.

During the past several decades,  many FD methods have been developed to solve acoustic wave equations\cite{alford1974accuracy,tam1993dispersion,yang2012central,liu2009implicit}.
Due to their high accuracy and simple implementation , the high-order FD methods have  attracted great  interests of many Applied Mathematicians and Geophysicists, 
who put a great deal of efforts in the analysis and development of  high-order FD methods for acoustic wave equations. Many  efficient high-order methods have 
been developed and implemented with great success\cite{liu2009new, dablain1986application,zingg1996high, zhang2017new,cohen1996construction,takeuchi2000optimally,oliveira2003fourth}, to name a few.

It is worth to notice that high-order methods  require less grid points \cite{etgen2007computational} and is effective in  minimizing dispersion error \cite{finkelstein2007finite}. It was also reported that high-order scheme allow a more coarse spatial sampling rate \cite{levander1988fourth}. 

However, many high-order FDM are not compact, which leads to difficulty in dealing with boundary conditions. For example, a typical 4th-order FDM requires a 5-point stencil to approximate $\Delta u$ in 1D cases, a 9-point stencil for 2D cases and 13-point stencil for 3D cases, respectively. Therefore, it is hard to implement if there is only one layer of boundary conditions. 

To overcome this issue, Compact FD method has been developed  to reduce the points needed to approximate derivatives but keep the high-order accuracy of FDM \cite{shukla2005derivation}. In \cite{lele1992compact} the author introduced a family of FDM to approximate the first and second derivatives, and those schemes can be compact with suitable choices of parameters. Some more information about recent compact FDM can be found in \cite{shukla2005derivation}.  More other reported  works
on compact high-order FD methods for solving acoustic wave equation can be found in \cite{kim1996optimized,liao2014dispersion, britt2018high} and references therein.

 Furthermore, many existing compact high-order FD methods were developed based on constant velocity model, thus,  fail in variable sound speed cases. Some schemes attempt to address this drawback.  For example, using a novel algebraic manipulation of finite difference operators, a high-order compact FDM was proposed in \cite{liao2018efficient} for 2D acoustic wave equations with variable velocity. Then a similar compact scheme based on Pad\'{e} approximation was developed in \cite{li2018compacthighorder} to solve 3D problem. 
Recently, in \cite{britt2018high} a new compact high-order method based on the efficient numerical solution of Helmholtz equation has been proposed to solve acoustic wave equation with variable  velocity. Another problem is that the stability analysis of the compact high-order FD method is difficult  using the standard Von Neumann stability analysis, which works only for problems with constant velocity. Very recently, an energy method was introduced to derive the stability condition for variable coefficients case \cite{britt2018high,li2018compacthighorder}.

In this paper, a new explicit compact scheme with 2nd-order temporal and 4th-order spatial accuracy has been proposed and fully studied. This new method overcomes the variable acoustic velocity problem and is very efficient and easy to implement even when the acoustic velocity is discontinuous. The boundary conditions needed in the new scheme are obtained from the equation itself exactly, while the initial condition at ghost time level for the solving process can be approximated by the equation with high accuracy. The stability condition is obtained by an energy method. The rest of the paper is organized as the follows. In Section 2 the new compact high order numerical scheme is derived, which is followed by a rigorous proof of the stability in Section 3. In Section 4 several methods are reviewed to improve the temporal accuracy to 4th-order. Three numerical examples are solved by the new method in Section 5, which is followed by a conclusion in Section 6.


\section{Derivation of the Compact High-order Scheme}\label{4th-order in Space}
In this paper the 3D acoustic wave equation is considered
\begin{equation}\label{3dacousticwave}
u_{tt} = \nu^2(x,y,z)\Delta u + s(t,x,y,z),\ (t,x,y,z)\in [0,T]\times \Omega
\end{equation}
with initial conditions and Dirichlet boundary conditions. Here $\nu (x,y,z)$ is the acoustic velocity, $s(t,x,y,z)$ the source function and $\Omega$ the computational domain.

The key to obtain high-order compact FD scheme for solving equation (\ref{3dacousticwave}) is to approximation $\Delta u$ with high order compact finite difference approximation.
In \cite{lele1992compact, chu1998three} a general scheme of high-order approximation of second derivative was proposed, which leads to the following so-called
combined compact 4th-order difference approximation of the second derivative
\begin{equation}\label{lele}
	a_1 v''_{i-1} + a_0 v''_i + a_1 v''_{i+1} = \frac{1}{h_x^2}
	\big(b_1 v_{i-1}+b_0 v_i + b_1 v_{i+1}\big)
\end{equation}
where $a_0 = 1$, $a_1 = \frac{1}{10}$, $b_0 = -\frac{12}{5}$, $b_1 = \frac{6}{5}$, and $h_x$ is the grid size in $x$.  In this paper the idea of combined compact finite difference scheme will be extended to
acoustic wave equation with variable velocity.

For the sake of simplicity, assume that $\Omega$ is a 3D rectangular box defined as
\begin{equation*}
	\Omega = [x_{min},x_{max}]\times[y_{min},y_{max}]\times[z_{min},z_{max}],
\end{equation*}
which is discretized into an $(N_x+2) \times (N_y+2) \times (N_z+2)$ grid with spatial grid sizes $h_x = \dfrac{x_{max} - x_{min}}{N_x + 1}$, $h_y = \dfrac{y_{max} - y_{min}}{N_y +1}$ and $h_z = \dfrac{z_{max} - z_{min}}{N_z + 1}$. Then the initial-boundary value problem of the 3D acoustic wave equation can be rewritten in this form
\begin{equation}\label{acoustic eq}
	\begin{cases}
	u_{tt} = \nu^2(x,y,z)\Delta u + s(t,x,y,z),\\
	u|_{t=0} = \alpha(x,y,z),\ u_t|_{t=0} = \beta(x,y,z),\\
	u|_{x=x_{min}} = f_0(t,y,z),\ u|_{x=x_{max}} = f_1(t,y,z), \\
	u|_{y=y_{min}} = g_0(t,x,z),\ u|_{y=y_{max}} = g_1(t,x,z), \\
	u|_{z=z_{min}} = h_0(t,x,y),\ u|_{z=z_{max}} = h_1(t,x,y). \\
	\end{cases}
\end{equation}

Denoted by $\tau$ the time step, $u_{i,j,k}^n$ the numerical solution at grid point $(x_i,y_j,z_k)=(x_{min} +i h_x,y_{min}+j h_y,z_{min}+ k h_z)$ and time level $t_n = n \tau$. The temporal derivatives can be approximated by the standard 2nd-order centre difference scheme. For the spatial derivatives one has the following 4th-order approximation
\begin{equation}
	\begin{split}
	& a_1 (u_{xx})_{i-1,j,k}^n + a_0 (u_{xx})_{i,j,k}^n + a_1 (u_{xx})_{,i+1,j,k}^n \\
	=& \frac{1}{h_x^2}
	\big(b_1 u_{i-1,j,k}^n+b_0 u_{i,j,k}^n + b_1 u_{i+1,j,k}^n\big),
	\end{split}
\end{equation}
\begin{equation}
	\begin{split}
		& a_1 (u_{yy})_{i,j-1,k}^n + a_0 (u_{yy})_{i,j,k}^n + a_1 (u_{yy})_{i,j+1,k}^n \\
		=& \frac{1}{h_y^2}
		\big(b_1 u_{i,j-1,k}^n+b_0 u_{i,j,k}^n + b_1 u_{i,j+1,k}^n\big),
	\end{split}
\end{equation}
\begin{equation}
\begin{split}
& a_1 (u_{zz})_{i,j,k-1}^n + a_0 (u_{zz})_{i,j,k}^n + a_1 (u_{zz})_{i,j,k+1}^n \\
=& \frac{1}{h_z^2}
\big(b_1 u_{i,j,k-1}^n+b_0 u_{i,j,k}^n + b_1 u_{i,j,k+1}^n\big),
\end{split}
\end{equation}
for $1\leqslant i\leqslant N_x$, $1\leqslant j\leqslant N_y$, $1\leqslant k\leqslant N_z$ where $(u_{xx})_{i,j,k}^n$, $(u_{yy})_{i,j,k}^n$, $(u_{zz})_{i,j,k}^n$ are the approximation of the second derivatives of $u$. Define the following vectors
\begin{equation}
	u_{*,j,k}^n =
	\begin{pmatrix}
	u_{1,j,k}^n\\
	u_{2,j,k}^n\\
	\vdots\\
	u_{N_x,j,k}^n
	\end{pmatrix}_{N_x \times 1},\ 
	(u_{xx})_{*,j,k}^n =
	\begin{pmatrix}
	(u_{xx})_{1,j,k}^n\\
	(u_{xx})_{2,j,k}^n\\
	\vdots\\
	(u_{xx})_{N_x,j,k}^n
	\end{pmatrix}_{N_x \times 1},
\end{equation}
also define $u_{i,*,k}$, $(u_{yy})_{i,*,k}$, $u_{i,j,*}$, $(u_{zz})_{i,j,*}$ in similar way. Then the approximation can be written in vector form
\begin{equation}\label{u_xx}
	A_x (u_{xx})_{*,j,k}^n+a_1q_x^{a,n} = \frac{1}{h_x^2} \left(B_x u_{*,j,k}^n+b_1q_x^{b,n} \right)
\end{equation}
where
\begin{equation}
	A_x = 
	\begin{pmatrix}
	a_0 & a_1 \\
	a_1 & a_0 & a_1\\
		& \dots &\dots\\
		& a_1 & a_0 & a_1\\
		&	  & a_1 & a_0
	\end{pmatrix}_{N_x \times N_x},\ 
	B_x = 
	\begin{pmatrix}
	b_0 & b_1 \\
	b_1 & b_0 & b_1\\
	& \dots &\dots\\
	& b_1 & b_0 & b_1\\
	&	  & b_1 & b_0
	\end{pmatrix}_{N_x \times N_x}
\end{equation}
are tridiagonal matrices, and
\begin{equation}
	q_x^{a,n} =
	\begin{pmatrix}
	(u_{xx})_{0,j,k}^n \\
	0\\
	\vdots\\
	0\\
	(u_{xx})_{N_x+1,j,k}^n
	\end{pmatrix}_{N_x \times 1},\ 
	q_x^{b,n} =
	\begin{pmatrix}
	u_{0,j,k}^n \\
	0\\
	\vdots\\
	0\\
	u_{N_x+1,j,k}^n
	\end{pmatrix}_{N_x \times 1}
\end{equation}
represent the boundary values. In a similar way one can rewrite the approximation of $u_{yy}$ and $u_{zz}$ in vector form
\begin{equation}\label{u_yy}
A_y (u_{yy})_{i,*,k}^n+a_1q_y^{a,n} = \frac{1}{h_y^2} \left ( B_y u_{i,*,k}^n+b_1q_y^{b,n}\right),
\end{equation}
\begin{equation}\label{u_zz}
A_z (u_{zz})_{i,j,*}^n+a_1q_z^{a,n} =\frac{1}{h_z^2} \left(B_z u_{i,j,*}^n+b_1q_z^{b,n}\right).
\end{equation}
Now the boundary values of $u_{i,j,k}^n$ are already known. For the boundary values of $u_{xx}$, $u_{yy}$ and $u_{zz}$, one can obtain them from the equation (\ref{acoustic eq}). For example, for $u_{xx}$
\begin{equation}\label{u_xx 0}
	(u_{xx})_{0,j,k}^n =
	\left(\frac{u_{tt}}{\nu^2}-\frac{s}{\nu^2}-u_{yy}-u_{zz}
	\right)|_{0,j,k}^n
\end{equation}
\begin{equation}\label{u_xx 1}
	(u_{xx})_{N_x+1,j,k}^n =
	\left(\frac{u_{tt}}{\nu^2}-\frac{s}{\nu^2}-u_{yy}-u_{zz}
	\right)|_{N_x+1,j,k}^n.
\end{equation}
Note that $u|_{x=x_{min}} = f_0(t,y,z)$ and $u|_{x=x_{max}} = f_1(t,y,z)$, then one has
\begin{equation}
(u_{tt})_{0,j,k}^n = (\partial_t^2 u)|_{x=x_{min}} = \partial_t^2 (u|_{x=x_{min}}) = (\partial_t^2 f_0)(t_n,y_j,z_k)
\end{equation}
\begin{equation}
(u_{tt})_{N_x+1,j,k}^n = (\partial_t^2 u)|_{x=x_{max}} = \partial_t^2 (u|_{x=x_{max}}) = (\partial_t^2 f_1)(t_n,y_j,z_k)
\end{equation}
\begin{equation}
	(u_{yy})_{0,j,k}^n = (\partial_y^2 u)|_{x=x_{min}} = \partial_y^2 (u|_{x=x_{min}}) = (\partial_y^2 f_0)(t_n,y_j,z_k)
\end{equation}
\begin{equation}
	(u_{yy})_{N_x+1,j,k}^n = (\partial_y^2 u)|_{x=x_{max}} = \partial_y^2 (u|_{x=x_{max}}) = (\partial_y^2 f_1)(t_n,y_j,z_k)
\end{equation}
\begin{equation}
(u_{zz})_{0,j,k}^n = (\partial_z^2 u)|_{x=x_{min}} = \partial_z^2 (u|_{x=x_{min}}) = (\partial_z^2 f_0)(t_n,y_j,z_k)
\end{equation}
\begin{equation}
(u_{zz})_{N_x+1,j,k}^n = (\partial_z^2 u)|_{x=x_{max}} = \partial_z^2 (u|_{x=x_{max}}) = (\partial_z^2 f_1)(t_n,y_j,z_k).
\end{equation}
Substitute the above equations into (\ref{u_xx 0})(\ref{u_xx 1}), one obtains
\begin{equation}
	(u_{xx})_{0,j,k}^n = \frac{1}{\nu^2_{0,j,k}}\big[(\partial_t^2 f_0)_{j,k}^n - s_{0,j,k}^n\big] - (\partial_y^2 f_0)_{j,k}^n - (\partial_z^2 f_0)_{j,k}^n
\end{equation}
\begin{equation}
(u_{xx})_{N_x+1,j,k}^n = \frac{1}{\nu^2_{N_x+1,j,k}}\big[(\partial_t^2 f_1)_{j,k}^n - s_{N_x+1,j,k}^n\big] - (\partial_y^2 f_1)_{j,k}^n - (\partial_z^2 f_1)_{j,k}^n.
\end{equation}
The boundary values of $u_{yy}$ and $u_{zz}$ can be obtained in a similar way. Now the linear systems (\ref{u_xx})(\ref{u_yy})(\ref{u_zz}) can be solved. Note that the matrices $A_x$, $A_y$ and $A_z$ are tridiagonal matrices, thus the linear systems can be solved by Thomas Algorithm in $O(N_x)$, $O(N_y)$ and $O(N_z)$ complexity, respectively. But one has to solve $N_y\times N_z$ linear systems for $u_{xx}$, $N_x \times N_z$ linear systems for $u_{yy}$ and $N_x \times N_y$ linear systems for $u_{zz}$, thus the overall complexity is $O(N_x N_y N_z)$ for each time step. In summary one has
\begin{equation}\label{u_xx eq}
	(u_{xx})_{*,j,k}^n =A_x^{-1} \left[\frac{1}{h_x^2}\left(B_x u_{*,j,k}^n+b_1q_x^{b,n} \right)- a_1q_x^{a,n}\right]
\end{equation}
\begin{equation}\label{u_yy eq}
(u_{yy})_{i,*,k}^n =A_y^{-1} \left[\frac{1}{h_y^2} \left(B_y u_{i,*,k}^n+b_1q_y^{b,n} \right)- a_1q_y^{a,n}\right]
\end{equation}
\begin{equation}\label{u_zz eq}
(u_{zz})_{i,j,*}^n =A_z^{-1} \left[\frac{1}{h_z^2} \left(B_z u_{i,j,*}^n+b_1q_z^{b,n} \right)- a_1q_z^{a,n}\right]
\end{equation}
i.e. the values of $(u_{xx})_{i,j,k}^n$, $(u_{yy})_{i,j,k}^n$ and $(u_{zz})_{i,j,k}^n$ for $1\leqslant i \leqslant N_x$, $1\leqslant j \leqslant N_y$ and $1\leqslant k \leqslant N_z$.

The equation (\ref{acoustic eq}) also yields 
\begin{equation}
(\partial_t^2 u)|_{t=0} = \nu^2 \Delta u|_{t=0} + s|_{t=0} = \nu^2 \Delta \alpha + s|_{t=0},
\end{equation}
\begin{equation}
(\partial_t^3 u)|_{t=0} = \nu^2 \Delta (\partial_t u)|_{t=0} + (\partial_t s)|_{t=0} = \nu^2 \Delta \beta + (\partial_t s)|_{t=0},
\end{equation}
where $\alpha = u|_{t=0}$ and $\beta = u_t|_{t=0}$, then one has
\begin{equation}
(u_{tt})^0_{i,j,k} = 
\nu^2_{i,j,k}(\Delta \alpha)_{i,j,k}+s^n_{i,j,k},
\end{equation}
\begin{equation}
	(u_{ttt})^0_{i,j,k} = 
	\nu^2_{i,j,k}(\Delta \beta)_{i,j,k}+(\partial_t s)^0_{i,j,k}.
\end{equation}
Thus $u^{-1}_{i,j,k}$, which denotes the numerical approximation at the ghost time level $t = -\tau$, can be approximated by
\begin{equation}\label{u at -tau}
	\begin{split}
	u^{-1}_{i,j,k} =& u^0_{i,j,k} - \tau (u_t)^0_{i,j,k} + \frac{1}{2}{\tau}^2 (u_{tt})^0_{i,j,k} - \frac{1}{6}{\tau}^3 (u_{ttt})^0_{i,j,k} + O(\tau^4)\\
	=& \alpha_{i,j,k} - \tau \beta_{i,j,k} + \frac{1}{2}{\tau}^2[\nu^2_{i,j,k}(\Delta \alpha)_{i,j,k}+s^0_{i,j,k}] \\
	-&\frac{1}{6}{\tau}^3[\nu^2_{i,j,k}(\Delta \beta)_{i,j,k}+(\partial_t s)^0_{i,j,k}]+O(\tau^4).
	\end{split}
\end{equation}

Finally, the following schemes with error $O(\tau^2)+O(h_x^4)+O(h_y^4)+O(h_z^4)$ is obtained
\begin{equation}
	u^{n+1}_{i,j,k} = \tau^2 [\nu^2_{i,j,k}(\Delta u)^{n}_{i,j,k}+s^n_{i,j,k} ]+2 u^{n}_{i,j,k} - u^{n-1}_{i,j,k}, \ n = 0, 1,2,\cdots
\end{equation}
with $u^{-1}_{i,j,k}$ from (\ref{u at -tau}) and $(\Delta u)^{n}_{i,j,k}$ from (\ref{u_xx eq})(\ref{u_yy eq})(\ref{u_zz eq}). A complete scheme with details is included in Appendix \ref{The Complete Scheme}.

\begin{remark}
	Note that if there exists two layers of boundary conditions, then an analogy of (\ref{lele}) can lead to 8th-order spatial accuracy.
\end{remark}

\section{Stability Analysis}
The stability analysis employs an energy method which is valid in both constant and variable coefficients cases. Consider the acoustic equation with zero boundary conditions and zero source term
\begin{equation}
	u_{tt} = \nu^2(x,y,z) \Delta u.
\end{equation}
For simplicity assume $h_x = h_y = h_z=h$ and $N_x = N_y = N_z=N$. Recall the 4th-order spatial approximation
\begin{equation}\label{u_xx stability}
\begin{split}
& a_1 (u_{xx})_{i-1,j,k}^n + a_0 (u_{xx})_{i,j,k}^n + a_1 (u_{xx})_{i+1,j,k}^n \\
=& \frac{1}{h^2}
\big(b_1 u_{i-1,j,k}^n+b_0 u_{i,j,k}^n + b_1 u_{i+1,j,k}^n\big)
\end{split}
\end{equation}
\begin{equation}\label{u_yy stability}
\begin{split}
& a_1 (u_{yy})_{i,j-1,k}^n + a_0 (u_{yy})_{i,j,k}^n + a_1 (u_{yy})_{i,j+1,k}^n \\
=& \frac{1}{h^2}
\big(b_1 u_{i,j-1,k}^n+b_0 u_{i,j,k}^n + b_1 u_{i,j+1,k}^n\big)
\end{split}
\end{equation}
\begin{equation}\label{u_zz stability}
\begin{split}
& a_1 (u_{zz})_{i,j,k-1}^n + a_0 (u_{zz})_{i,j,k}^n + a_1 (u_{zz})_{i,j,k+1}^n \\
=& \frac{1}{h^2}
\big(b_1 u_{i,j,k-1}^n+b_0 u_{i,j,k}^n + b_1 u_{i,j,k+1}^n\big).
\end{split}
\end{equation}
If one let
\begin{equation}\label{matrix A and B}
A = 
\begin{pmatrix}
a_0 & a_1 \\
a_1 & a_0 & a_1\\
& \dots &\dots\\
& a_1 & a_0 & a_1\\
&	  & a_1 & a_0
\end{pmatrix}_{N \times N},\ 
B = 
\begin{pmatrix}
b_0 & b_1 \\
b_1 & b_0 & b_1\\
& \dots &\dots\\
& b_1 & b_0 & b_1\\
&	  & b_1 & b_0
\end{pmatrix}_{N \times N}
\end{equation}
and $U^n$ be the vector form of the numerical solution $u^n_{i,j,k}$
\begin{equation}
	U^n =
	\begin{pmatrix}
	u^n_{1,1,1} & \dots & u^n_{N,1,1} &u^n_{1,2,1} & \dots & u^n_{N,2,1} & u^n_{1,3,1} & \dots & u^n_{N,3,1} & \dots
	\end{pmatrix}^T
\end{equation}
i.e. $U^n$ is an $N^3\times 1$ vector in which $u^n_{i,j,k}$ is located at the $(k N^2 + j N + i)$-th row. Also define $U_{xx}^n$, $U_{yy}^n$ and $U_{zz}^n$ as the vector forms of the second derivatives in a similar way. Then it is straightforward to see that the equations (\ref{u_xx stability})(\ref{u_yy stability})(\ref{u_zz stability}) can be written in the following form
\begin{equation}
	(A\otimes I_N \otimes I_N) U_{xx}^n = (B\otimes I_N \otimes I_N)U^n
\end{equation}
\begin{equation}
(I_N\otimes A \otimes I_N) U_{yy}^n = (I_N\otimes B \otimes I_N)U^n
\end{equation}
\begin{equation}
(I_N\otimes I_N \otimes A) U_{zz}^n = (I_N\otimes I_N \otimes B)U^n
\end{equation}
where $\otimes$ indicates Kronecker product and $I_N$ is the $N\times N$ identity matrix. Then one has
\begin{equation}
 U_{xx}^n =(A\otimes I_N \otimes I_N)^{-1} (B\otimes I_N \otimes I_N)U^n,
\end{equation}
\begin{equation}
 U_{yy}^n = (I_N\otimes A \otimes I_N)^{-1}(I_N\otimes B \otimes I_N)U^n,
\end{equation}
\begin{equation}
 U_{zz}^n = (I_N\otimes I_N \otimes A)^{-1}(I_N\otimes I_N \otimes B)U^n.
\end{equation}
\begin{lemma}\label{Kronecker property}
	Kronecker product is associative. The followings hold
	\begin{equation}
		\begin{split}
		I_m \otimes I_n &= I_{mn}\\
		(A\otimes B)(C\otimes D) &= (AC)\otimes (BD)\\
		(A\otimes B)^{-1}&= (A^{-1}\otimes B^{-1})\\
		(A\otimes B)^{T}&= (A^{T}\otimes B^{T})
		\end{split}
	\end{equation}
\end{lemma}
\begin{lemma}\label{Kronecker eig}
	Any eigenvalue of $A \otimes B$ arises as a product of eigenvalues of $A$ and $B$.
\end{lemma}
The above two lemmas can be found in \cite{horn1991topics}. By those lemmas, one has
\begin{equation}
	\begin{split}
	A_1 :&=(A\otimes I_N \otimes I_N)^{-1} (B\otimes I_N \otimes I_N) = (A^{-1}B\otimes I_N \otimes I_N)\\
	A_2 :&=(I_N\otimes A \otimes I_N)^{-1} (B\otimes I_N \otimes I_N) = (I_N\otimes A^{-1}B \otimes I_N)\\
	A_3 :&=(I_N\otimes I_N \otimes A)^{-1} (B\otimes I_N \otimes I_N) = (I_N\otimes I_N \otimes A^{-1}B)\\
	\end{split}
\end{equation}
and 
\begin{equation*}
	\sigma(A_1)=\sigma(A_2)=\sigma(A_3)=\sigma(A^{-1}B)
\end{equation*}
where $\sigma$ denotes the spectrum.

Let $\circ$ denote the entrywise product of matrices, i.e. if $A=(a_{ml})_{q\times q}$ and $B=(b_{ml})_{q\times q}$ are two matrices then
\begin{equation}
	A\circ B = (a_{ml}b_{ml})_{q\times q}.
\end{equation}
Now one can write the numerical scheme in Section \ref{4th-order in Space} for acoustic equation with zero boundary condition and zero source term in the following form
\begin{equation}\label{U eq 0}
	\delta_t^2 U^n = \frac{\tau^2}{h^2}\cdot C \circ (A_1 U^n +A_2 U^n +A_3 U^n)
\end{equation}
where $\delta_t$ is the difference operator, $C$ the vector form of $\nu^2_{i,j,k}$. The scheme (\ref{U eq 0}) is equivalent to the following form
\begin{equation}\label{U eq}
	\Phi \circ \delta_t^2 U^n =  (A_1 U^n +A_2 U^n +A_3 U^n):= LU^n
\end{equation}
where $\Phi$ is the vector form of $\frac{h^2}{\tau^2} \frac{1}{\nu^2_{i,j,k}}$.

In order to prove the stability of the scheme (\ref{U eq}), it is necessary to estimate the spectrum of $L=A_1+A_2+A_3$.
\begin{lemma}
	$A_1$, $A_2$ and $A_3$ are self-adjoint matrices, and one has
	\begin{equation}
		\sigma(A_1)=\sigma(A_2)=\sigma(A_3)=\sigma(A^{-1}B) \subset (-6,-r(N)]
	\end{equation}
	and $L$ is self-adjoint with
	\begin{equation}
		\sigma (L) \subset (-18,-3r(N)].
	\end{equation}
	Here
	\begin{equation}
		r(N) = \left(1+\frac{1}{5}\cos \frac{\pi}{N+1}\right)^{-1}\cdot\left(\frac{12}{5} - \frac{12}{5}\cos \frac{\pi}{N+1}\right)> 0.
	\end{equation}
\end{lemma}
\begin{proof}
	Note that both $A$ and $B$ are symmetric tridiagonal Toeplitz matrices, then they share common eigenvectors, see \cite{smith1985numerical,noschese2013tridiagonal}. Note that $A$ has distinct eigenvalues, as well as $B$ does. On the other hand, $A^{-1}$ has the same eigenvectors as $A$. Thus $A^{-1}$ and $B$ commute because of common eigenvectors and distinct eigenvalues. Also note that $A^{-1}$ is symmetric since $A$ is. Thus $A^{-1}$ and $B$ are commutative and both symmetric. Then $A^{-1}B$ is symmetric. Thus $A_1$, $A_2$ and $A_3$ are all self-adjoint by Lemma \ref{Kronecker property}.
	
	The spectrum of $A$ and $B$ are given by
	\begin{equation}
		\sigma (A)  = \left\{1+\frac{1}{5}\cos\left(\frac{\pi l}{N+1}\right) \right\}, \ l=1,\dots,N
	\end{equation}
	\begin{equation}
		\sigma (B) = \left\{-\frac{12}{5}+\frac{12}{5}\cos\left(\frac{\pi l}{N+1}\right)	\right\}, \ l=1,\dots,N
	\end{equation}
	Then one can estimate the spectrum of $A^{-1}B$ by (see \cite{horn1991topics})
	\begin{equation}
		\sigma(A^{-1}B)\subset 
		\left[-\frac{\frac{12}{5} + \frac{12}{5}\cos \frac{\pi}{N+1}}{1-\frac{1}{5}\cos \frac{\pi}{N+1}},
		-\frac{\frac{12}{5} - \frac{12}{5}\cos \frac{\pi}{N+1}}{1+\frac{1}{5}\cos \frac{\pi}{N+1}}
		\right]
		\subset (-6,-r(N)].
	\end{equation}
	Then by Lemma \ref{Kronecker eig}, one has
	\begin{equation}
		\sigma(A_1)=\sigma(A_2)=\sigma(A_3)=\sigma(A^{-1}B)\subset (-6,-r(N)].
	\end{equation}
	Finally $L$ is self-adjoint since it is a sum of self-adjoint matrices and the spectrum of $L$ is estimated by (see \cite{knutson2001honeycombs})
	\begin{equation}
	\sigma(L)\subset (-18,-3r(N)].
	\end{equation}
\end{proof}
Thus a coercive condition for the operator $L$ is obtained
\begin{equation}\label{coercive}
	0<m = 3r(N) \leqslant -L \leqslant 18 = M.
\end{equation}
Now one can obtain the main result on the stability
\begin{theorem}
	The new scheme in Section \ref{4th-order in Space} is stable if
	\begin{equation}
		\max_{1\leqslant i,j,k \leqslant N} \frac{\nu_{i,j,k}\cdot\tau}{h} < \frac{\sqrt{2}}{3}.
	\end{equation}
\end{theorem}
\begin{proof}
	The idea comes from \cite{britt2018high}. Denoted by $\| \cdot\|$ the $L^2$ norm, and $\langle\cdot,\cdot\rangle$ the $L^2$ inner product. From the above discussion, the new scheme is equivalent to
	\begin{equation}\label{U in proof}
		\Phi \circ \delta_t^2 U^n = LU^n
	\end{equation}
	where $\Phi$ is the vector form of $\frac{h^2}{\tau^2} \frac{1}{\nu^2_{i,j,k}}$.
	
	Let $\Gamma^n = U^n - U^{n-1}$, then $\delta_t^2 U^n = \Gamma^{n+1} - \Gamma^n$. Now the equation (\ref{U in proof}) can be written as
	\begin{equation}\label{Gamma in proof}
		\Phi \circ (\Gamma^{n+1} - \Gamma^n) = LU^n.
	\end{equation}
	Consider the $L^2$ inner product of both sides of (\ref{Gamma in proof}) with $U^{n+1} - U^{n-1} = \Gamma^{n+1}+\Gamma^n$
	\begin{equation}\label{inner product}
		\langle \Phi \circ (\Gamma^{n+1} - \Gamma^n),U^{n+1} - U^{n-1}\rangle = \langle LU^n,U^{n+1} - U^{n-1}\rangle.
	\end{equation}
	Since $L$ is self-adjoint, the right-hand-side of (\ref{inner product}) can be expanded as
	\begin{equation}\label{RHS inner product}
		\begin{split}
		\langle LU^n,U^{n+1} - U^{n-1}\rangle 
		=& \frac{1}{4}\big[\langle L\Gamma^n,\Gamma^n \rangle - \langle L (U^n+U^{n-1}),(U^n+U^{n-1})\rangle
		\big] \\
		-& \frac{1}{4}\big[\langle L\Gamma^{n+1},\Gamma^{n+1} \rangle - \langle L (U^{n+1}+U^{n}),(U^{n+1}+U^{n})\rangle
		\big].
		\end{split}
	\end{equation}
	For the left-hand-side of (\ref{inner product}) one has
	\begin{equation}\label{LHS inner product}
		\begin{split}
		&\langle \Phi \circ (\Gamma^{n+1} - \Gamma^n),U^{n+1} - U^{n-1}\rangle \\
		 =& \langle \Phi \circ (\Gamma^{n+1} - \Gamma^n),\Gamma^{n+1} + \Gamma^n\rangle \\
		 =& \langle \Phi \circ \Gamma^{n+1},\Gamma^{n+1} \rangle - \langle \Phi \circ \Gamma^{n},\Gamma^{n} \rangle.
		\end{split}
	\end{equation}
	Let
	\begin{equation}
		R^{n} = \langle \Phi \circ \Gamma^{n},\Gamma^{n} \rangle + \frac{1}{4}\langle L \Gamma^{n},\Gamma^{n}\rangle - \frac{1}{4}\langle L (U^{n}+U^{n-1}),(U^{n}+U^{n-1})\rangle
	\end{equation}
	then (\ref{inner product})(\ref{LHS inner product})(\ref{RHS inner product}) result in the identity
	\begin{equation}
		R^{n+1} = R^{n}.
	\end{equation}
	Recall the coercive condition of $L$ in (\ref{coercive}), one has
	\begin{equation}
		R^n \geqslant \Phi_{min} \|\Gamma^n\|^2 - \frac{M}{4}\|\Gamma^n\|^2 + \frac{m}{4} \|U^{n}+U^{n-1}\|^2,
	\end{equation}
	and
	\begin{equation}
	R^n \leqslant \Phi_{max} \|\Gamma^n\|^2 - \frac{m}{4}\|\Gamma^n\|^2 + \frac{M}{4} \|U^{n}+U^{n-1}\|^2.
	\end{equation}
	Thus if
	\begin{equation}\label{3.37}
		\Phi_{min} - \frac{M}{4} > 0
	\end{equation}
	then $R^n$ is equivalent to the energy given by
	\begin{equation}
		\|\Gamma^n\|^2 + \|U^{n}+U^{n-1}\|^2 = 2\|U^n\|^2 +2\|U^{n-1}\|^2.
	\end{equation}
	Since
	\begin{equation}
		\Phi_{min} = \frac{h^2}{\tau^2}\cdot \frac{1}{(\nu^2_{i,j,k})_{max}}
	\end{equation}
	then (\ref{3.37}) becomes
	\begin{equation}\label{stability square}
		\max_{i,j,k} \nu_{i,j,k}^2 \cdot \frac{\tau^2}{h^2} < \frac{4}{M} = \frac{4}{18}
	\end{equation}
	i.e.
	\begin{equation}\label{stability}
		\max_{i,j,k} \nu_{i,j,k} \cdot \frac{\tau}{h} < \frac{\sqrt{2}}{3}.
	\end{equation}
	In this case, denoted by $e^n$ the error at time step $t_n$, the above stability analysis shows that the energy of the error $\|e^n\|^2 + \|e^{n-1}\|^2$ conserves during the solving process, which means the scheme is stable if (\ref{stability}) is satisfied.
\end{proof}
\begin{remark}
	The scheme
	\begin{equation*}
		\frac{1}{\Delta_x^2}\big(\frac{6}{5}v_{i-1} - \frac{12}{5}v_i + \frac{6}{5}v_{i+1}\big) = \frac{1}{10} v''_{i-1} + v''_{i} + \frac{1}{10} v''_{i+1}
	\end{equation*}
	can be written as
	\begin{equation*}
		\frac{1}{\Delta_x^2}\delta_x^2 v = \big(1+\frac{1}{12}\big)\delta_x^2 v''
	\end{equation*}
	which is formally
	\begin{equation*}
		\frac{1}{\Delta_x^2} \frac{\delta_x^2}{1+\frac{1}{12}\delta_x^2} v=v''.
	\end{equation*}
	Thus the stability analysis in this section is consistent with the result in \cite{li2018compacthighorder}
\end{remark}

\section{Higher Order Temporal Accuracy}
The new scheme in Section \ref{4th-order in Space} is 4th-order accurate in space. However, the total truncation error is only of 2nd-order in time as $O(\tau^2)+O(h_x^4)+O(h_y^4)+O(h_z^4)$. In general, it is desirable to get higher order accuracy in time as well \cite{chen2007high}. This section reviews several methods that can improve the temporal accuracy from 2nd-order $O(\tau^2)$ to 4th-order $O(\tau^4)$. 

Pad\'{e} approximation replaces the 2nd-order centre difference in time $\frac{1}{\tau^2}\delta_t^2$ by the 4th-order approximation $\frac{1}{\tau^2}\cdot \frac{\delta_t^2}{1+\frac{1}{12}\delta_t^2}$, then the scheme becomes
\begin{equation}\label{Pade 0}
	\frac{\delta_t^2}{1+\frac{1}{12}\delta_t^2} u^{n}_{i,j,k} = \tau^2 [\nu^2_{i,j,k}(\Delta u)^{n}_{i,j,k}+s^n_{i,j,k}].
\end{equation}
Multiply the both sides of (\ref{Pade 0}) by $1+\frac{1}{12}\delta_t^2$, one obtains
\begin{equation}\label{Pade 1}
	\delta_t^2 u^{n}_{i,j,k} = \tau^2 [\nu^2_{i,j,k}(1+\frac{1}{12}\delta_t^2)(\Delta u)^{n}_{i,j,k}+(1+\frac{1}{12}\delta_t^2)s^n_{i,j,k}].
\end{equation}
It has been proved in \cite{li2018compacthighorder} that the scheme (\ref{Pade 1}) has a slightly better CFL constant, which comes from the fact that Pad\'{e} approximation in time $\frac{1}{\tau^2}\cdot \frac{\delta_t^2}{1+\frac{1}{12}\delta_t^2}$ improves the constant $\frac{4}{M}$ in (\ref{stability square}) to $\frac{6}{M}$ which leads to a CFL constant $\frac{\sqrt{3}}{3}$. However, the trade off is that the scheme (\ref{Pade 1}) is implicit which is expensive to solve for $u_{i,j,k}^{n+1}$ directly. An iterative method is needed there to solve the large sparse linear system.

One can also apply Richardson Extrapolation to improve the temporal accuracy. Denoted by $NS(T;\tau,h)$, the numerical solution of (\ref{acoustic eq}) solved by the new scheme with grid size $h$ and time step $\tau$, evaluated at $t = T$. Then the Richardson Extrapolation of the numerical solution evaluated at $t=T$ is given by
\begin{equation}
RE(T;\tau,h) = \frac{2^l NS(T;\frac{1}{2}\tau,h)-NS(T;\tau,h)}{2^l - 1},\ l = 2,3,4,\cdots.
\end{equation}

The Runge-Kutta Method is a different approach which works only for a system of  first oder ordinary differential equations. Therefore, one has to rewrite the equation (\ref{acoustic eq}) as a first-order system, then apply and explicit Runge-Kutta method and obtain
\begin{equation}
\begin{cases}
\mathcal{U}^{n+1} = \mathcal{U}^{n}+ \tau \sum_{l=1}^{q} b_l \mathcal{K}_l\\
\mathcal{K}_l =\mathcal{A}\left(\mathcal{U}^n + \tau\sum_{p=1}^{q-1}a_{lp}\mathcal{K}_p\right) + \mathcal{S}(t_n + c_l \tau) \\
\end{cases}
\end{equation}
with
\begin{equation}
\mathcal{U} = 
\begin{pmatrix}
u \\ u_t
\end{pmatrix},\ 
\mathcal{A} = 
\begin{pmatrix}
0 & 1\\
\nu^2 \Delta & 0\\
\end{pmatrix},\ 
\mathcal{S} = 
\begin{pmatrix}
0 \\ s(t,x,y,z)
\end{pmatrix},\ 
\mathcal{K}_l = 
\begin{pmatrix}
K_{l,1} \\ K_{l,2}
\end{pmatrix}
\end{equation}
and some constants $a_{lp}$, $b_l$, $c_l$. Note that those constants are different from what appear in Section \ref{4th-order in Space}.
%
Here $\Delta \mathcal{K}_{l,1}$ is approximated by the same method as (\ref{u_xx eq})(\ref{u_yy eq})(\ref{u_zz eq}) to ensure 4th-order
accuracy in space. Thus the boundary value $\mathcal{K}_{l-1}|_{\Omega}$ is necessary, which can be obtained by the same method used for $(\Delta u)|_{\Omega}$ in (\ref{u_xx 0}), see Appendix \ref{Appendix:RK4}.
One numerical example has been solved to show that both Richardson extrapolation and Runge-Kutta methods can improve the order of accuracy to 4th-order. It is worthy to mention that extra caution should be taken when the two methods are used, as they might bring in some stability issues.

\section{Numerical Experiments}
In this section, three numerical experiments are conducted with the new scheme to demonstrate the efficiency and accuracy. The first example verifies the 2nd-order in time and 4th-order in space accuracy. The second example compares the performance of Richardson Extrapolation and Runge-Kutta Methods in improving temporal accuracy to 4th-order. The third example considers a more realistic problem, solving an underground acoustic model with Ricker wavelet source. 
\subsection{Example 1}\label{e.g.:1}
This example solves the acoustic wave equation defined on the domain $\Omega = [0,1]\times[0,1]\times[0,1]$ and $t \in [0, T]$
\begin{equation}
	u_{tt} = \nu^2 \Delta u + s,
\end{equation}
where
\begin{equation}
	\nu^2 = \frac{1}{(x-\frac{1}{2})(y-\frac{1}{2})(z-\frac{1}{2})+\frac{1}{6}}
\end{equation}
and
\begin{equation}
	s = (4-14\nu^2)e^{2t}e^{x+2y+3z}
\end{equation}
with initial and Dirichlet boundary conditions compatible to the analytical solution which is given by
\begin{equation}
	u = e^{2t}e^{x+2y+3z}.
\end{equation}
In order to validate that the new scheme is  2nd-order in time and 4th-order in space, let $h_x = h_y = h_z = h$, and set the grid size $h$ and time step $\tau$ to satisfy $\tau = h^2$. The errors in max norm and energy norm with different $h$ are listed in Table \ref{table:1}, which shows that the new scheme is indeed of error $O(\tau^2) + O(h^4)$. Here the convergence order is calculated by 
\begin{equation}
	\text{Conv. Order} = \frac{\log\big[E(h_1)/E(h_2)\big]}{\log(h_1/h_2)}.
\end{equation}
\begin{remark}
	Since $\tau = h^2$ in this case, it is actually verified that the new scheme is of accuracy $O(h^4)$. Note that the number of time steps is of order $O(\tau^{-1}) = O(h^{-2})$, and the approximation of $u(-\tau)$ is of order $O(\tau^4) = O(h^8)$, see (\ref{u at -tau}), thus the aggregate error from $u(-\tau)$ is of order $O(h^{-2}h^{8}) = O(h^6) \ll O(h^4)$.
\end{remark}
\begin{table}
	\centering
	\caption{Numerical errors in max norm and energy norm for Example \ref{e.g.:1} with $\tau = h^2$}
	\begin{tabular}{|c|c|c|c|c|}
		\hline
		h & 1/10 & 1/15 & 1/20 & 1/25  \\
		\hline
		$E_{max}$ & 0.0047& 9.5748e-04 & 3.0609e-04 & 1.2605e-04 	\\
		\hline
		$E_{energy}$ & 0.0020 & 3.8709e-04 & 1.1990e-04 & 4.8467e-05	\\
		\hline
		Conv. Order (max)	& - & 3.9239 & 3.9642 & 3.9759 	\\
		\hline
		Conv. Order (energy)	& - & 4.0503 & 4.0739 & 4.0592	\\
		\hline
	\end{tabular}

	\label{table:1}
\end{table}
\subsection{Example 2}\label{e.g.:compare RK and RE}
This example compares Richardson Extrapolation (RE) and Runge-Kutta Method (RK) in increasing temporal accuracy to 4th-order. Consider the 3D 
acoustic wave equation defined  on $\Omega = [0,1]\times[0,1]\times[0,1]$ and  $t \in [0, T]$.  
\begin{equation}\label{eq:compare}
	\begin{cases}
	u_{tt} = (1+xyz)\Delta u + s(t,x,y,z) \\
	u|_{t=0} = \sin(\pi x)\sin(\pi y)\sin(\pi z)\\
	u_t|_{t=0} = \pi \sin(\pi x)\sin(\pi y)\sin(\pi z)\\
	u|_{\partial \Omega} = 0
	\end{cases}
\end{equation}
with source
\begin{equation}
	s(t,x,y,z) = (4+3xyz)\pi^2 e^{\pi t}\sin(\pi x)\sin(\pi y)\sin(\pi z).
\end{equation}
The analytical solution is given by
\begin{equation}
	u = e^{\pi t}\sin(\pi x)\sin(\pi y)\sin(\pi z).
\end{equation}
For simplicity,  let $h_x = h_y = h_z = h$ and $\tau = \frac{h}{10}$ in all test cases.
Denoted by $NS(T;\tau,h)$, the numerical solution of (\ref{eq:compare}) solved by the new scheme with grid size $h$ and time step $\tau$, evaluated at $t = T$. Then the Richardson Extrapolation of the numerical solution evaluated at $t=T$ is given by
\begin{equation}
	RE(T;\tau,h) = \frac{4 NS(T;\frac{1}{2}\tau,h)-NS(T;\tau,h)}{3}.
\end{equation}
Rewrite the equation (\ref{eq:compare}) as a first-order system
\begin{equation}
	\mathcal{U}_t = \mathcal{A}\mathcal{U} + \mathcal{S}(t)
\end{equation}
with 
\begin{equation}
\mathcal{U} = 
\begin{pmatrix}
u \\ u_t
\end{pmatrix},\ 
\mathcal{A} = 
\begin{pmatrix}
0 & 1\\
(1+xyz) \Delta & 0\\
\end{pmatrix},\ 
\mathcal{S}(t) = 
\begin{pmatrix}
0 \\ s(t)
\end{pmatrix},\ 
\mathcal{K}_l = 
\begin{pmatrix}
K_{l,1} \\ K_{l,2}
\end{pmatrix}.
\end{equation}
Then one can consider the conventional 4th-order Runge-Kutta Method (RK4)
\begin{equation}
	\begin{cases}
	\mathcal{U}^{n+1} = \mathcal{U}^{n} + \frac{\tau}{6}\left(\mathcal{K}_1 + \frac{1}{2}\mathcal{K}_2 + \frac{1}{2}\mathcal{K}_3 + \mathcal{K}_4 \right)\\
	\mathcal{K}_1 = \mathcal{A}\mathcal{U}^n + \mathcal{S}(t_n)\\
	\mathcal{K}_2 = \mathcal{A}\left(\mathcal{U}^n + \frac{\tau}{2}\mathcal{K}_1\right) + \mathcal{S}(t_n + \frac{1}{2}\tau)\\
	\mathcal{K}_3 = \mathcal{A}\left(\mathcal{U}^n + \frac{\tau}{2}\mathcal{K}_2\right) + \mathcal{S}(t_n + \frac{1}{2}\tau)\\
	\mathcal{K}_4 = \mathcal{A}\left(\mathcal{U}^n + \tau\mathcal{K}_3\right) + \mathcal{S}(t_n + \tau)
	\end{cases}.
\end{equation}
Both RE and RK4 should improve the temporal accuracy from $O(\tau^2)$ to $O(\tau^4)$, which leads to an overall accuracy $O(h^4)$. Denoted by $E_{RE}$ the error of RE, and $E_{RK4}$ the error of RK4, Table \ref{table:compare} compares those two methods in energy norm.
\begin{table}
	\centering
	\caption{Comparison of RE and RK4 in Example \ref{e.g.:compare RK and RE} in energy norm, with $\tau = \frac{h}{10}$.}
	\begin{tabular}{|c|c|c|c|c|}
		\hline
		h & 1/10 & 1/15 & 1/20 & 1/25  \\
		\hline
		$E_{RE}$ & 2.9340e-04 & 5.4765e-05 & 1.6862e-05 & 6.7968e-06	\\
		\hline
		$E_{RK4}$ & 2.9327e-04 & 5.4734e-05 & 1.6852e-05 & 6.7924e-06	\\
		\hline
		RE Conv. Order	& - & 4.1397 & 4.0948 & 4.0719	\\
		\hline
		RK4 Conv. Order	& - & 4.1400 & 4.0949& 4.0721	\\
		\hline
		RE CPU time (s) & 0.175715 & 0.674665 & 1.678161 & 3.572069 \\
		\hline
		RK4 CPU time (s) & 0.271528 & 0.966194 & 2.412038 & 5.024197 \\
		\hline
	\end{tabular}
	\label{table:compare}
\end{table}
From Table \ref{table:compare} one can see that RK4 has slightly better accuracy than RE, but it spends about 50\% more CPU time. The reason is stated below. Comparing to the original scheme RE does an additional solving process for $\frac{\tau}{2}$, which leads to an overall CPU time about 3 times as much as the original one. However, RK4 used here requires additionally approximating four $\Delta \mathcal{K}_l$'s in each time step which leads to an overall CPU time about 5 times as much as the original one.

\subsection{Example 3}\label{e.g.:ricker's}
This example solves a more realistic problem in which the seismic wave is generated by a Ricker wavelet source. The region is a three-dimensional domain $\Omega = [0m,1200m]\times[0m,1200m]\times[0m,1350m]$. The velocity is given by
\begin{equation}
	\nu (x,y,z) = 
	\begin{cases}
	1200\ m/s,\text{if } 0 \leqslant z \leqslant 879.75 \\
	2500\ m/s,\text{if } 879.75 < z \leqslant 1350
	\end{cases}.
\end{equation}
This can be regarded as that the region is divided in two parts. From the ground surface to $879.75m$ underground is soil with sound speed $\nu = 1200m/s$, and from $879.75m$ underground to $1350m$ underground is rock with sound speed $\nu = 2500m/s$. The underground model is sketched in Figure \ref{fig:underground}. The Ricker wavelet source is given by
\begin{equation}
	s(t,x,y,z) = \delta (x-x_s,y-y_s,z-z_s)[1-2\pi^2 f^2_p (t-d_r)^2]e^{-\pi^2 f_p^2 (t-d_r)^2}
\end{equation}
with dominant frequency $f_p = 10 Hz$, temporal delay $d_r = 0.5/f_p$. The wave generator is located at $(x_s,y_s,z_s) = (600m,600m,600m)$, which is the source location in the soil area. The time step $\tau = 0.0005s$ and grid size $h = 5m$ are chosen to satisfy Nyquist Theorem on spatial resolution and the stability condition
given in  (\ref{stability})
\begin{equation}
	\nu_{max} \frac{\tau}{h} < \frac{\sqrt{2}}{3}.
\end{equation}
\begin{figure}[h]
	\includegraphics[width=4cm]{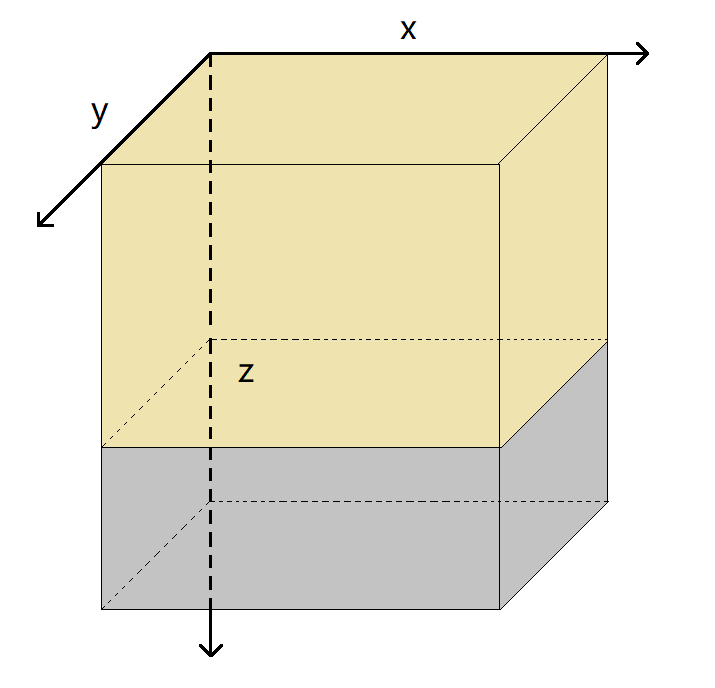}
	\caption{Underground model for Example \ref{e.g.:ricker's}, the yellow part is soil with sound speed $\nu = 1200m/s$, the gray part is rock with sound speed $\nu = 2500m/s$.}
	\label{fig:underground}
\end{figure}

Four snapshots of y-section at $y = y_s$ are plotted. Note that the distance between the source and the rock area is $879.75m - 600m = 279.75m$, thus the wave will reach the rock area at $t = \frac{279.75}{1200}s = 0.233125s$. Figure \ref{fig:wave1} shows that at $t = 0.225s$ the y-section of the wave is still a circle. Figure \ref{fig:wave2} shows that at $t = 0.375s$ the wave has reached the rock area, thus both reflection and refraction occur when the wave crosses the interface of the two media at $z = 897.75m$. Figure \ref{fig:042} shows the wave has been reflected back from bottom boundary at $t = 0.42s$. Figure \ref{fig:066} shows the reflected waves at $t = 0.66s$. Note that on Figure \ref{fig:066}, the wave reflected from bottom boundary has crossed the interface of soil and rock, thus it shows a different shape from Figure \ref{fig:042} due to the refraction.
\begin{figure}[h]
	\includegraphics[width=12cm]{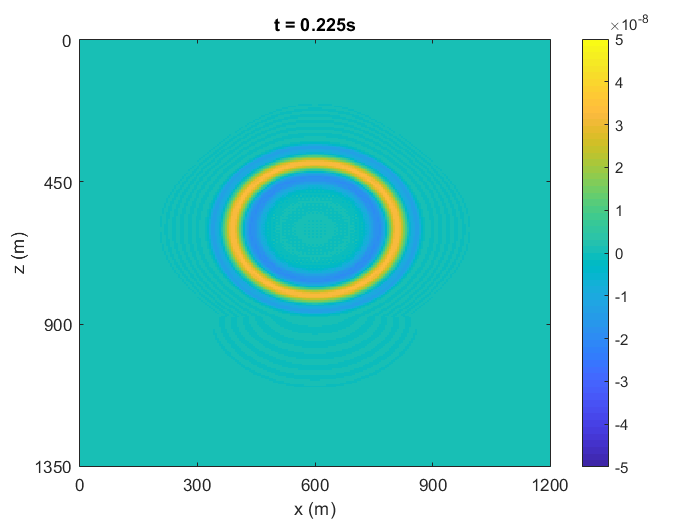}
	\caption{Snapshot of y-section at $t = 0.225s$ in Example \ref{e.g.:ricker's}}
	\label{fig:wave1}
\end{figure}
\begin{figure}[h]
	\includegraphics[width=12cm]{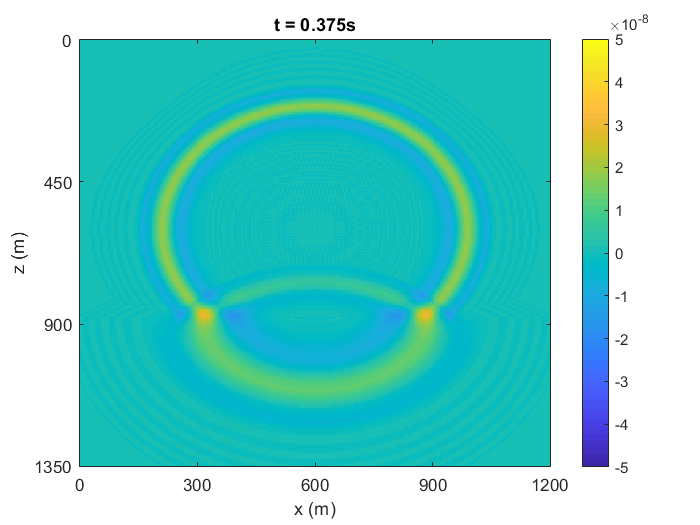}
	\caption{Snapshot of y-section at $t = 0.375s$ in Example \ref{e.g.:ricker's}}
	\label{fig:wave2}
\end{figure}
\begin{figure}[h]
	\includegraphics[width=12cm]{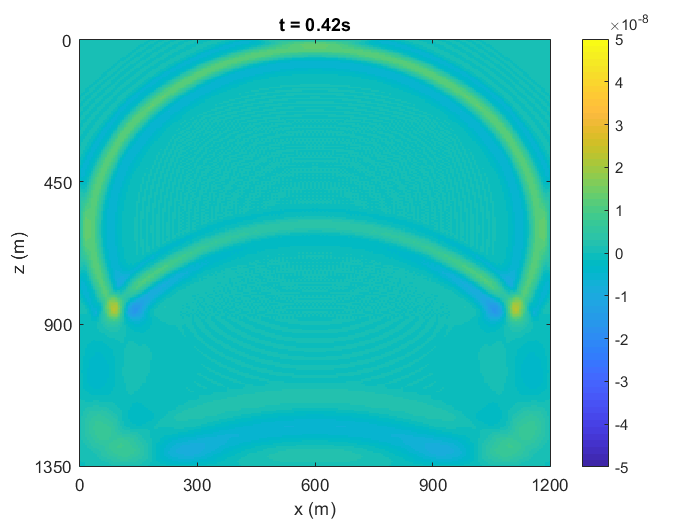}
	\caption{Snapshot of y-section at $t = 0.42s$ in Example \ref{e.g.:ricker's}}
	\label{fig:042}
\end{figure}
\begin{figure}[h]
	\includegraphics[width=12cm]{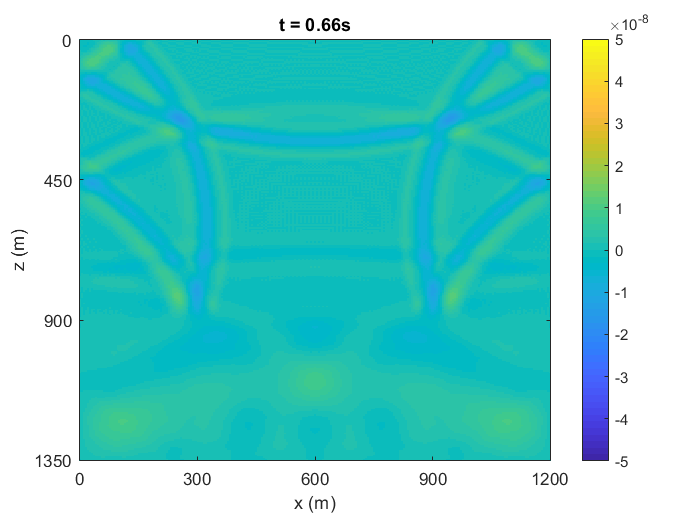}
	\caption{Snapshot of y-section at $t = 0.66s$ in Example \ref{e.g.:ricker's}}
	\label{fig:066}
\end{figure}

There are also some spurious waves propagating faster than the true wave on those figures. One may have to zoom in the figures to find them. They are inconspicuous and neglectable comparing to the true wave. Their occurrence is resulted from the way how $\Delta u$ is approximated in this new method. For simplicity consider an interval $\Omega = [-2,2]$ and a smooth function $v$ supported on $[-1,1]$. The second derivative on $\Omega$ is approximated by
\begin{equation}\label{error}
v_{xx} = A^{-1}Bv
\end{equation}
where $A$ and $B$ are given by (\ref{matrix A and B}).  The matrix $A^{-1}$ will spread a little part of $v$ to its second derivative outside of the support $[-1,1]$, which makes the approximation of $v_{xx}$ non-zero on $[-2,-1]$ and $[1,2]$.

\section{Conclusion}
In this work a compact explicit 2nd-order in time and 4th-order in space FDM has been developed to solve acoustic equations with variable acoustic velocity. The new scheme has a time complexity which is linear to the total number of  grid points  for each time step. This new scheme is conditionally stable with a slightly lower CFL condition. Two numerical experiments are conducted for the new scheme to validate the convergence order and the compatibility with Richardson Extrapolation and 4th-order Runge-Kutta Method. A more realistic problem has been solved as the third numerical example  to show that the method is accurate and efficient for numerical simulation of acoustic wave propagation in 3D heterogeneous media. The new scheme is expected to find wide applications in numerical seismic modelling and related areas.

In the future, the authors plan to generalize this new scheme to acoustic wave equations with spatial variable density, in which the Laplacian $\Delta u$ is replaced by $\nabla \cdot \left(\frac{1}{\rho}\nabla\right)u$, with density $\rho = \rho(x,y,z)$. Moreover, more boundary conditions such as absorbing
boundary and perfectly matched layer (PML) will be considered.
\appendix
\section{Details of the New Scheme}\label{The Complete Scheme}
The numerical scheme for equation (\ref{acoustic eq}) in Section \ref{4th-order in Space} is given by
\begin{equation}
u^{n+1}_{i,j,k} = \tau^2 [\nu^2_{i,j,k}(\Delta u)^{n}_{i,j,k}+s^n_{i,j,k} ]+2 u^{n}_{i,j,k} - u^{n-1}_{i,j,k}
\end{equation}
for $n\geqslant 0$ and $1\leqslant i\leqslant N_x$, $1\leqslant j\leqslant N_y$, $1\leqslant k\leqslant N_z$, with
\begin{equation}
	\begin{split}
	u^{-1}_{i,j,k} = & \alpha_{i,j,k} - \tau \beta_{i,j,k} + \frac{\tau^2}{2}[\nu^2_{i,j,k}(\Delta \alpha)_{i,j,k}+s^0_{i,j,k}] \\
	- &\frac{{\tau}^3}{6} [\nu^2_{i,j,k}(\Delta \beta)_{i,j,k} + ({\partial}_t s)_{i,j,k}^0 ] +O(\tau^4)
	\end{split}
\end{equation}
and
\begin{equation}\label{u_xx Appendix}
(u_{xx})_{*,j,k}^n =A_x^{-1} \left[\frac{1}{h_x^2} \left( B_x u_{*,j,k}^n+b_1q_x^{b,n} \right)- a_1q_x^{a,n}\right]
\end{equation}
\begin{equation}\label{u_yy Appendix}
(u_{yy})_{i,*,k}^n =A_y^{-1} \left[\frac{1}{h_y^2} \left(B_y u_{i,*,k}^n+b_1q_y^{b,n} \right)- a_1q_y^{a,n}\right]
\end{equation}
\begin{equation}\label{u_zz Appendix}
(u_{zz})_{i,j,*}^n =A_z^{-1} \left[\frac{1}{h_z^2} \left(B_z u_{i,j,*}^n+b_1q_z^{b,n} \right)- a_1q_z^{a,n}\right]
\end{equation}
where
\begin{equation}
A_\lambda = 
\begin{pmatrix}
a_0 & a_1 \\
a_1 & a_0 & a_1\\
& \dots &\dots\\
& a_1 & a_0 & a_1\\
&	  & a_1 & a_0
\end{pmatrix}_{N_\lambda \times N_\lambda},\ 
B_\lambda = 
\begin{pmatrix}
b_0 & b_1 \\
b_1 & b_0 & b_1\\
& \dots &\dots\\
& b_1 & b_0 & b_1\\
&	  & b_1 & b_0
\end{pmatrix}_{N_\lambda \times N_\lambda}
\end{equation}
for $\lambda \in \{x,y,z\}$ and
\begin{equation}
u_{*,j,k}^n =
\begin{pmatrix}
u_{1,j,k}^n\\
u_{2,j,k}^n\\
\vdots\\
u_{N_x,j,k}^n
\end{pmatrix}_{N_x \times 1},\ 
(u_{xx})_{*,j,k}^n =
\begin{pmatrix}
(u_{xx})_{1,j,k}^n\\
(u_{xx})_{2,j,k}^n\\
\vdots\\
(u_{xx})_{N_x,j,k}^n
\end{pmatrix}_{N_x \times 1},
\end{equation}
\begin{equation}
u_{i,*,k}^n =
\begin{pmatrix}
u_{i,1,k}^n\\
u_{i,2,k}^n\\
\vdots\\
u_{i,N_y,k}^n
\end{pmatrix}_{N_y \times 1},\ 
(u_{yy})_{i,*,k}^n =
\begin{pmatrix}
(u_{yy})_{i,1,k}^n\\
(u_{yy})_{i,2,k}^n\\
\vdots\\
(u_{yy})_{i,N_y,k}^n
\end{pmatrix}_{N_y \times 1},
\end{equation}
\begin{equation}
u_{i,j,*}^n =
\begin{pmatrix}
u_{i,j,1}^n\\
u_{i,j,2}^n\\
\vdots\\
u_{i,j,N_z}^n
\end{pmatrix}_{N_z \times 1},\ 
(u_{zz})_{i,j,*}^n =
\begin{pmatrix}
(u_{zz})_{i,j,1}^n\\
(u_{zz})_{i,j,2}^n\\
\vdots\\
(u_{zz})_{i,j,N_z}^n
\end{pmatrix}_{N_z \times 1}.
\end{equation}
The $q$-vectors for boundary values in (\ref{u_xx Appendix})(\ref{u_yy Appendix})(\ref{u_zz Appendix}) are given by
\begin{equation}
q_x^{a,n} =
\begin{pmatrix}
(u_{xx})_{0,j,k}^n \\
0\\
\vdots\\
0\\
(u_{xx})_{N_x+1,j,k}^n
\end{pmatrix}_{N_x \times 1},\ 
q_x^{b,n} =
\begin{pmatrix}
(f_0)_{j,k}^n \\
0\\
\vdots\\
0\\
(f_1)_{j,k}^n
\end{pmatrix}_{N_x \times 1},
\end{equation}
\begin{equation}
q_y^{a,n} =
\begin{pmatrix}
(u_{yy})_{i,0,k}^n \\
0\\
\vdots\\
0\\
(u_{yy})_{i,N_y+1,k}^n
\end{pmatrix}_{N_y \times 1},\ 
q_y^{b,n} =
\begin{pmatrix}
(g_0)_{i,k}^n \\
0\\
\vdots\\
0\\
(g_1)_{i,k}^n
\end{pmatrix}_{N_y \times 1},
\end{equation}
\begin{equation}
q_z^{a,n} =
\begin{pmatrix}
(u_{zz})_{i,j,0}^n \\
0\\
\vdots\\
0\\
(u_{zz})_{i,j,N_z +1}^n
\end{pmatrix}_{N_z \times 1},\ 
q_x^{b,n} =
\begin{pmatrix}
(h_0)_{i,j}^n \\
0\\
\vdots\\
0\\
(h_1)_{i,j}^n
\end{pmatrix}_{N_z \times 1}.
\end{equation}
The boundary values for $u_{xx}$, $u_{yy}$ and $u_{zz}$ above are given by
\begin{equation}
(u_{xx})_{0,j,k}^n = \frac{1}{\nu^2_{0,j,k}}\big[(\partial_t^2 f_0)_{j,k}^n - s_{0,j,k}^n\big] - (\partial_y^2 f_0)_{j,k}^n - (\partial_z^2 f_0)_{j,k}^n
\end{equation}
\begin{equation}
(u_{xx})_{N_x+1,j,k}^n = \frac{1}{\nu^2_{N_x+1,j,k}}\big[(\partial_t^2 f_1)_{j,k}^n - s_{N_x+1,j,k}^n\big] - (\partial_y^2 f_1)_{j,k}^n - (\partial_z^2 f_1)_{j,k}^n
\end{equation}
\begin{equation}
(u_{yy})_{i,0,k}^n = \frac{1}{\nu^2_{i,0,k}}\big[(\partial_t^2 g_0)_{i,k}^n - s_{i,0,k}^n\big] - (\partial_x^2 g_0)_{i,k}^n - (\partial_z^2 g_0)_{i,k}^n
\end{equation}
\begin{equation}
(u_{yy})_{i,N_y+1,k}^n = \frac{1}{\nu^2_{i,N_y+1,k}}\big[(\partial_t^2 g_1)_{i,k}^n - s_{i,N_y+1,k}^n\big] - (\partial_x^2 g_1)_{i,k}^n - (\partial_z^2 g_1)_{i,k}^n
\end{equation}
\begin{equation}
(u_{zz})_{i,j,0}^n = \frac{1}{\nu^2_{i,j,0}}\big[(\partial_t^2 h_0)_{i,j}^n - s_{i,j,0}^n\big] - (\partial_x^2 h_0)_{i,j}^n - (\partial_y^2 h_0)_{i,j}^n
\end{equation}
\begin{equation}
(u_{zz})_{i,j,N_z+1}^n = \frac{1}{\nu^2_{i,j,N_z+1}}\big[(\partial_t^2 h_1)_{i,j}^n - s_{i,j,N_z+1}^n\big] - (\partial_x^2 h_1)_{i,j}^n - (\partial_y^2 h_1)_{i,j}^n.
\end{equation}

\section{Example on Boundary Values in Runge-Kutta Methods}\label{Appendix:RK4}
This section briefly shows how to deal with the boundary values in implementing Runge-Kutta Methods by an example. To ensure 4th-order spatial accuracy, all of the Laplacian $\Delta$ appears below should be considered as approximation by the same method as (\ref{u_xx eq})(\ref{u_yy eq})(\ref{u_zz eq}). Thus some boundary values of $\mathcal{K}$'s are necessary.

Consider the following RK4
\begin{equation}
\begin{cases}
\mathcal{U}^{n+1} = \mathcal{U}^{n} + \frac{\tau}{6}\left(\mathcal{K}_1 + \frac{1}{2}\mathcal{K}_2 + \frac{1}{2}\mathcal{K}_3 + \mathcal{K}_4 \right)\\
\mathcal{K}_1 = \mathcal{A}\mathcal{U}^n + \mathcal{S}(t_n)\\
\mathcal{K}_2 = \mathcal{A}\left(\mathcal{U}^n + \frac{\tau}{2}\mathcal{K}_1\right) + \mathcal{S}(t_n + \frac{1}{2}\tau)\\
\mathcal{K}_3 = \mathcal{A}\left(\mathcal{U}^n + \frac{\tau}{2}\mathcal{K}_2\right) + \mathcal{S}(t_n + \frac{1}{2}\tau)\\
\mathcal{K}_4 = \mathcal{A}\left(\mathcal{U}^n + \tau\mathcal{K}_3\right) + \mathcal{S}(t_n + \tau)
\end{cases}
\end{equation}
where
\begin{equation}
\mathcal{U} = 
\begin{pmatrix}
u \\ u_t
\end{pmatrix},\ 
\mathcal{A} = 
\begin{pmatrix}
0 & 1\\
\nu^2 \Delta & 0\\
\end{pmatrix},\ 
\mathcal{S} = 
\begin{pmatrix}
0 \\ s
\end{pmatrix},\ 
\mathcal{K}_l = 
\begin{pmatrix}
K_{l,1} \\ K_{l,2}
\end{pmatrix}.
\end{equation}
One will have
\begin{equation}
	\begin{split}
	K_{11} = & u_t|^{t_n} \\
	K_{12} = & \nu^2 \Delta u|^{t_n} + s|^{t_n}\\
	\end{split}
\end{equation}
\begin{equation}
	\begin{split}
	K_{21} = & u_t|^{t_n} + \frac{\tau}{2}K_{12} \\
	K_{22} = & \nu^2 \Delta u|^{t_n} + \frac{\tau}{2}\nu^2 \Delta K_{11} + s|^{t_n+\frac{\tau}{2}}\\
	\end{split}
	\end{equation}
	\begin{equation}
	\begin{split}
	K_{31} = & u_t|^{t_n} + \frac{\tau}{2}K_{22} \\
	K_{32} = & \nu^2 \Delta u|^{t_n} + \frac{\tau}{2}\nu^2 \Delta K_{21} + s|^{t_n+\frac{\tau}{2}}\\
	\end{split}
	\end{equation}
	\begin{equation}
	\begin{split}
	K_{41} = & u_t|^{t_n} + \tau K_{32} \\
	K_{42} = & \nu^2 \Delta u|^{t_n} + \tau \nu^2 \Delta K_{31} + s|^{t_n+\tau}\\
	\end{split}
\end{equation}
Thus the boundary values of $K_{11}$, $K_{21}$, $K_{31}$ are required to approximate $\Delta K_{11}$, $\Delta K_{21}$, $\Delta K_{31}$, respectively. Check the expressions of $K_{21}$ and $K_{31}$, one may find that the boundary values of $K_{12}$ and $K_{22}$ are also necessary.

By the equation (\ref{acoustic eq}), one has
\begin{equation}
	\begin{split}
	K_{11}|_{x_{min}} = & (\partial_t f_0)|^{t_n} \\
	K_{12}|_{x_{min}} = & \nu^2 \Delta u|^{t_n}_{x_{min}} + s|^{t_n}_{x_{min}} \\
	= & u_{tt}|^{t_n}_{x_{min}} = (\partial_t^2 f_0)|^{t_n}\\
	\end{split}
\end{equation}
Then $K_{21}|_{x_{min}}$ will be easy to obtained. Note that $(\nu^2 \Delta K_{11})|_{x_{min}}$ is necessary to evaluate $K_{22}|_{x_{min}}$
\begin{equation}
	\begin{split}
	(\nu^2 \Delta K_{11})|_{x_{min}} = & (\nu^2 \Delta u_t|^{t_n})|_{x_{min}} = \biggl\{ \left[\partial_t(\nu^2 \Delta u)\right]|^{t_n}\biggr\}|_{x_{min}} \\
	= & \biggl\{ \left[\partial_t(u_{tt} - s)\right]|^{t_n}\biggr\}|_{x_{min}} \\
	= & (\partial_t^3 f_0 - s|_{x_{min}})|^{t_n}
	\end{split}
\end{equation}
Finally, $K_{31}|_{x_{min}}$ will be easy to obtained as well. The boundary conditions at other boundaries can be obtained similarly.
\begin{remark}
	Note that
	\begin{equation}
		\begin{split}
		u_t|^{t_n}_{x_{min}} = & (\partial_t f_0)|^{t_n} \\
		\nu^2 \Delta u|^{t_n}_{x_{min}} = & u_{tt}|^{t_n}_{x_{min}} - s|^{t_n}_{x_{min}} = (\partial_t^2 f_0)|^{t_n} -s|^{t_n}_{x_{min}}. \\
		\end{split}
	\end{equation}
\end{remark}
\nocite{*}
\bibliography{BibFD4thOrder}
\bibliographystyle{siam}

\end{document}